\documentclass[11pt]{amsproc}

\usepackage{amssymb}
\usepackage{mathtools}
\usepackage[dvipsnames]{xcolor}
\usepackage{fullpage}
\usepackage[hypertexnames=false]{hyperref}
\usepackage{aliascnt}

\usepackage[capitalize,nameinlink,noabbrev,nosort]{cleveref}
\hypersetup{
  colorlinks=true,
  allcolors=CarnationPink
}

\crefname{claim}{Claim}{Claims}

\makeatletter
\patchcmd{\@maketitle}
  {\@settitle}
  {\vspace*{-2.5cm}\@settitle}
  {}{}
\makeatother

\theoremstyle{plain}
\newtheorem{theorem}{Theorem}[section]

\newaliascnt{proposition}{theorem}
\newtheorem{proposition}[proposition]{Proposition}
\aliascntresetthe{proposition}

\newaliascnt{lemma}{theorem}
\newtheorem{lemma}[lemma]{Lemma}
\aliascntresetthe{lemma}

\newaliascnt{claim}{theorem}
\newtheorem{claim}[claim]{Claim}
\aliascntresetthe{claim}

\numberwithin{equation}{section}

\newcommand{\bbF}{\mathbb{F}}
\newcommand{\bbZ}{\mathbb{Z}}
\newcommand{\bbQ}{\mathbb{Q}}

\DeclareMathOperator{\Res}{Res}
\begin{document}

\title{Irreducibility of truncations of the Catalan generating function}

\author{Shin-ichiro Seki}
\thanks{This research was supported by JSPS KAKENHI Grant Number JP26K06734.}
\address{Nagahama Institute of Bio-Science and Technology, 1266, Tamura, Nagahama, Shiga, 526-0829, Japan}
\email{s\_seki@nagahama-i-bio.ac.jp}
\begin{abstract}
We prove that, for every positive integer $n$, the polynomial
\[
C_0+C_1x+\cdots+C_nx^n
\]
is irreducible over $\bbQ$, where $C_k$ denotes the $k$th Catalan number.
\end{abstract}
\maketitle
\section{Introduction}
The \emph{Catalan numbers} are defined by
\[
C_k\coloneqq\frac{1}{k+1}\binom{2k}{k}
\]
for each nonnegative integer $k$:
\begin{gather*}
C_0=1,\quad C_1=1,\quad C_2=2,\quad C_3=5,\quad
C_4=14,\quad C_5=42,\quad C_6=132,\quad C_7=429,\\
C_8=1430,\quad C_9=4862,\quad C_{10}=16796,\quad
C_{11}=58786,\quad C_{12}=208012,\quad \dots.
\end{gather*}
Their generating function is given by
\[
\sum_{k=0}^{\infty}C_kx^k=\frac{1-\sqrt{1-4x}}{2x}
\]
and has radius of convergence $1/4$.
For each nonnegative integer $n$, we consider the degree $n$ truncation
\[
P_n(x)\coloneqq\sum_{k=0}^nC_kx^k\in\bbZ[x].
\]
In this paper, we prove the following.
\begin{theorem}\label{thm:main}
For every positive integer $n$, the polynomial $P_n(x)$ is irreducible over $\bbQ$.
\end{theorem}
This was conjectured by Zhi-Wei Sun on March 23, 2013, in OEIS \href{https://oeis.org/A000108}{A000108}.

There are other known examples of formal power series whose nonconstant truncations are all irreducible over $\bbQ$.
A classical example is the exponential series $\exp(x)$.
\begin{theorem}[Schur \cite{Schur1929}]
For every positive integer $n$, the polynomial
\[
E_n(x)\coloneqq1+x+\frac{x^2}{2!}+\cdots+\frac{x^n}{n!}
\]
is irreducible over $\bbQ$.
\end{theorem}
Schur not only proved the irreducibility of these polynomials but also determined their Galois groups.
\begin{theorem}[Schur \cite{Schur1930}]
For every positive integer $n$, the Galois group of $E_n(x)$ over $\bbQ$ is isomorphic to the alternating group $A_n$ if $n$ is divisible by $4$, and to the symmetric group $S_n$ otherwise.
\end{theorem}
Coleman \cite{Coleman} gave an alternative proof of this theorem using Newton polygons.

Katz and Rivin \cite{KatzRivin} used Magma to carry out systematic computations of the Galois groups of truncations of a wide range of formal
power series with rational coefficients.
Among other results, they verified computationally that the Galois group of $P_n(x)$ over $\bbQ$ is isomorphic to the symmetric group $S_n$ for $1\leq n\leq 2000$; see \cite[\S4]{KatzRivin}.
In particular, the irreducibility of $P_n(x)$ over $\bbQ$ was already known in this range.
Sun conjectured in OEIS \href{https://oeis.org/A224416}{A224416} that the Galois group of $P_n(x)$ over $\bbQ$ is isomorphic to $S_n$ for every positive integer $n$.
We do not address this stronger conjecture in the present paper.
\section{Nuts and bolts}
In this section, we prepare two tools for use later in the proof.
For each nonnegative integer $n$, define the monic polynomial $Q_n(x)\in\bbZ[x]$ by
\[
Q_n(x)\coloneqq x^nP_n(1/x)=\sum_{k=0}^nC_kx^{n-k}.
\]
Since $P_n(x)$ is irreducible over $\bbQ$ if and only if $Q_n(x)$ is, we work with $Q_n(x)$ in what follows.
\subsection{Location of the roots of $Q_n(x)$}
For each positive integer $n$, define
\[
\Delta_n\coloneqq\left(\frac{31}{8}\right)^{n+1}+\sum_{j=1}^n\frac{6C_{n-j}}{n-j+2}\left(\frac{31}{8}\right)^j.
\]
\begin{lemma}\label{lem:31/8}
For every $n\geq 320$, $\Delta_n<4C_n$.
\end{lemma}
\begin{proof}
For $n\geq 320$, we have $C_{n+1}/C_n\geq 641/161$, which implies
\[
\frac{31}{8}(8C_n-C_{n+1})<4C_{n+1}.
\]
Thus, the identity
\[
\Delta_{n+1}=\frac{31}{8}\Delta_n+\frac{31}{8}(4C_n-C_{n+1})
\]
shows that $\Delta_n<4C_n$ implies $\Delta_{n+1}<4C_{n+1}$.
Therefore, it remains to verify that $\Delta_{320}<4C_{320}$.
This can be checked numerically; in fact, $\Delta_{320}/(4C_{320})=0.99325275\dots$.
\end{proof}
This inequality does not hold for $n=319$.
In fact, $\Delta_{319}/(4C_{319})=1.01582917\dots$.
\begin{proposition}\label{prop:location}
Let $n\geq 320$.
Then every complex root $\alpha$ of $Q_n(x)$ satisfies
\[
\frac{31}{8}<|\alpha|\leq 4-\frac{6}{n+1}<4.
\]
\end{proposition}
\begin{proof}
The upper bound follows from the Enestr\"om--Kakeya theorem, since
\[
\frac{C_k}{C_{k-1}}=\frac{2(2k-1)}{k+1}=4-\frac{6}{k+1}\leq 4-\frac{6}{n+1}
\]
for $1\leq k\leq n$.
To prove the lower bound, suppose that $Q_n(x)$ has a root $\alpha$ with $|\alpha|\leq 31/8$.
Since
\[
(4-x)Q_n(x)=4C_n+\sum_{j=1}^n\frac{6C_{n-j}}{n-j+2}x^j-x^{n+1},
\]
we have
\[
4C_n=\alpha^{n+1}-\sum_{j=1}^n\frac{6C_{n-j}}{n-j+2}\alpha^j.
\]
By the triangle inequality, we obtain $4C_n\leq\Delta_n$, contradicting Lemma~\ref{lem:31/8}.
\end{proof}
\subsection{A lower bound for a product of primes}
For each positive integer $n$, define
\[
\Pi_n\coloneqq\prod_{\substack{n+1<p\leq 2n \\ p: \text{ prime}}}p.
\]
\begin{proposition}\label{prop:Pi_n-ineq}
For $n\geq 320$, we have $\Pi_n>2^{6n/5}$.
\end{proposition}
\begin{proof}
Let $\vartheta(x)$ denote the first Chebyshev function, defined by
\[
\vartheta(x)\coloneqq\sum_{p\leq x}\log p.
\]
Rosser and Schoenfeld \cite[Theorem 4]{RosserSchoenfeld} proved that
\[
|\vartheta(x)-x|\leq\frac{x}{2\log x}
\]
for $x\geq 563$.
Since $\log\Pi_n=\vartheta(2n)-\vartheta(n+1)$, we can use this estimate, together with some numerical computations, to obtain the desired inequality.
\end{proof}
\section{Reduction to a functional equation}\label{sec:reduction}
Since the irreducibility of $Q_n(x)$ has been verified computationally for $1\leq n<320$, it suffices to consider the case $n\geq 320$.
Suppose, for a contradiction, that $Q_n(x)$ is reducible over $\bbQ$.
By Gauss's lemma, we may write $Q_n(x)=F(x)G(x)$ with nonconstant monic polynomials $F(x),G(x)\in\bbZ[x]$.
The following claim plays a key role in the proof and will be proved
in the next section.
\begin{claim}[Forced functional equation]\label{claim}
After interchanging $F(x)$ and $G(x)$ if necessary, we have
\begin{equation}\label{eq:forced_fn_eq}
F(4-x)=(-1)^{\deg F}F(x).
\end{equation}
\end{claim}
Let $W(x)\coloneqq x^2-4x+16$, and let $\alpha$ be a root of $F$.
\cref{claim} implies that $4-\alpha$ is also a root of $F$.
Set $a\coloneqq 16-|\alpha|^2$ and $b\coloneqq 16-|4-\alpha|^2$.
By \cref{prop:location}, we have $0<a,b<1$.
A direct calculation shows that
\[
|W(\alpha)|^2=a^2-ab+b^2,
\]
and hence $|W(\alpha)|<1$.

Every root of $F(x)$ has absolute value less than $4$, whereas the roots
$2\pm 2\sqrt{3}i$ of $W(x)$ have absolute value $4$.
Thus, $F(x)$ and $W(x)$ have no common root.
Let $\alpha_1,\dots,\alpha_{\deg F}$ be the roots of $F$, counted with multiplicity (in fact, all roots are simple).
Since $F(x)$ is monic, we obtain
\[
0<|\Res(F,W)|=\prod_{i=1}^{\deg F}|W(\alpha_i)|<1.
\]
Here, $\Res(F,W)$ denotes the resultant of $F(x)$ and $W(x)$.
On the other hand, since $F(x), W(x)\in\bbZ[x]$, we have $\Res(F,W)\in\bbZ$, a contradiction.
This reduces the proof of \cref{thm:main} to \cref{claim}.
\section{Proof of the claim}
\subsection{Local functional equation}
For each prime $p$ and each polynomial $f(x)\in\bbZ[x]$, let $\overline{f}(x)\in\bbF_p[x]$ denote its reduction modulo $p$.
In this subsection, all roots are taken in an algebraic closure $\overline{\bbF_p}$ of $\bbF_p$.

Let $n\geq 320$, and let $p$ be a prime satisfying $n+1<p\leq 2n$.
Since $p\mid C_k$ for $\frac{p+1}{2}\leq k\leq n$, we have
\begin{equation}\label{eq:factorization}
\overline{Q_n}(x)=x^{n-\frac{p-1}{2}}\overline{Q_{\frac{p-1}{2}}}(x).
\end{equation}
For $0\leq k\leq \frac{p-1}{2}$, the congruence
\[
2C_k\equiv -(-4)^{k+1}\binom{\frac{p+1}{2}}{k+1}\pmod{p}
\]
holds, and hence
\[
2\overline{Q_{\frac{p-1}{2}}}(x)=x^{\frac{p+1}{2}}-(x-4)^{\frac{p+1}{2}}.
\]
This can also be deduced from Mattarei and Tauraso's result \cite[Theorem 3]{MattareiTauraso}.
It follows that
\[
\overline{Q_{\frac{p-1}{2}}}(x)
\bigl(x^{\frac{p+1}{2}}+(x-4)^{\frac{p+1}{2}}\bigr)=2(x^p+x-4).
\]
Therefore,
\[
\overline{Q_{\frac{p-1}{2}}}\mid x^p+x-4
\]
in $\bbF_p[x]$.
Since the derivative of $x^p+x-4$ is $1$, this polynomial has no multiple roots.
\begin{lemma}
Let $U(x)\in\bbF_p[x]$ be any monic divisor of $\overline{Q_{\frac{p-1}{2}}}(x)$.
Let $\gamma_1,\dots,\gamma_{\deg U}$ be the roots of $U(x)$.
Then
\begin{equation}\label{eq:local_fn_eq}
U(4-x)=(-1)^{\deg U}U(x)
\end{equation}
and
\begin{equation}\label{eq:loc_sum_of_roots}
\sum_{i=1}^{\deg U}\gamma_i=2\deg U.
\end{equation}
\end{lemma}
\begin{proof}
Let $\gamma$ be a root of $U$.
Since $U(x)\in\bbF_p[x]$, we have
\[
U(\gamma^p)=U(\gamma)^p=0.
\]
Moreover, $U(x)\mid x^p+x-4$ implies that $\gamma^p=4-\gamma$.
Thus, $4-\gamma$ is also a root of $U(x)$.
Since the map $\gamma\mapsto 4-\gamma$ is an involution, it permutes the roots of $U$.
Thus,
\begin{align*}
U(4-x)&=\prod_{i=1}^{\deg U}(4-x-\gamma_i)=(-1)^{\deg U}\prod_{i=1}^{\deg U}\bigl(x-(4-\gamma_i)\bigr)\\
&=(-1)^{\deg U}\prod_{i=1}^{\deg U}(x-\gamma_i)=(-1)^{\deg U}U(x).
\end{align*}
Moreover,
\[
\sum_{i=1}^{\deg U}\gamma_i=\sum_{i=1}^{\deg U}(4-\gamma_i)=4\deg U-\sum_{i=1}^{\deg U}\gamma_i.
\]
Since $p$ is odd, this gives \eqref{eq:loc_sum_of_roots}.
\end{proof}
\subsection{Lifting the local functional equations}
As in \cref{sec:reduction}, let $n\geq 320$, and suppose that $Q_n(x)=F(x)G(x)$ with nonconstant monic polynomials $F(x),G(x)\in\bbZ[x]$.
Let $\alpha_1,\dots,\alpha_{\deg F}$ and $\beta_1,\dots,\beta_{\deg G}$ be the complex roots of $F(x)$ and $G(x)$, respectively, counted with multiplicity.
Set
\[
\widehat{\deg}\,F\coloneqq 2\deg F-\sum_{i=1}^{\deg F}\alpha_i,\quad \widehat{\deg}\,G\coloneqq 2\deg G-\sum_{j=1}^{\deg G}\beta_j.
\]
Since $F(x)$ and $G(x)$ are monic with integer coefficients, their sums of roots, and hence $\widehat{\deg}\,F$ and $\widehat{\deg}\,G$, are integers.
\begin{claim}\label{claim:deg-equiv}
Let $p$ be a prime satisfying $n+1<p\leq 2n$.
If $p\nmid F(0)$, then
\[
\widehat{\deg}\,F\equiv 0\pmod{p},
\]
whereas if $p\mid F(0)$, then
\[
\widehat{\deg}\,F\equiv 2n+1\pmod{p}.
\]
The same assertion holds for $G$.
\end{claim}
\begin{proof}
If $p\nmid F(0)$, then \eqref{eq:factorization} implies that $\overline{F}(x)\mid\overline{Q_{\frac{p-1}{2}}}(x)$.
Hence, by \eqref{eq:loc_sum_of_roots}, we obtain
\[
\sum_{i=1}^{\deg F}\alpha_i
\equiv \text{the sum of the roots of }\overline{F}(x)
=2\deg\overline{F}=2\deg F\pmod{p},
\]
which gives $\widehat{\deg}\,F\equiv 0\pmod{p}$.

If $p\mid F(0)$, then $p\nmid G(0)$, since $F(0)G(0)=C_n$ is divisible by $p$ but not by $p^2$.
Thus, $x^{n-\frac{p-1}{2}}\mid\overline{F}(x)$, and we may write $\overline{F}(x)=x^{n-\frac{p-1}{2}}U(x)$, where $U(x)\in\bbF_p[x]$ is monic and satisfies $U(x)\mid\overline{Q_{\frac{p-1}{2}}}(x)$.
Since $\deg U=\deg F-n+\frac{p-1}{2}$, it follows from \eqref{eq:loc_sum_of_roots} that
\[
\sum_{i=1}^{\deg F}\alpha_i\equiv \text{the sum of the roots of }U(x)=2\deg U\equiv 2\deg F-2n-1\pmod{p},
\]
which gives $\widehat{\deg}\,F\equiv 2n+1\pmod{p}$.
\end{proof}
\cref{claim:deg-equiv} gives
\begin{equation}\label{eq:Pi_n-division}
\Pi_n\mid \widehat{\deg}\,F \ \bigl(2n+1-\widehat{\deg}\,F\bigr).
\end{equation}
After interchanging $F(x)$ and $G(x)$ if necessary, we may first assume
that $\deg F\leq n/2$.
Since $|\alpha_i|<4$ by \cref{prop:location}, we have
\[
\left|\sum_{i=1}^{\deg F}\alpha_i\right|<4\deg F\leq 2n.
\]
Thus, $-n<\widehat{\deg}\,F<3n$, and \cref{prop:Pi_n-ineq} yields
\[
|\widehat{\deg}\,F \ (2n+1-\widehat{\deg}\,F)|<3n(n+1)<2^{6n/5}<\Pi_n.
\]
Combining this estimate with the divisibility relation \eqref{eq:Pi_n-division}, we obtain
$\widehat{\deg}\,F \ (2n+1-\widehat{\deg}\,F)=0$.
Hence, either $\widehat{\deg}\,F=0$ or $\widehat{\deg}\,F=2n+1$.
Since $\deg F+\deg G=n$ and
\[
\sum_{i=1}^{\deg F}\alpha_i+\sum_{j=1}^{\deg G}\beta_j=-C_1=-1,
\]
we have $\widehat{\deg}\,F+\widehat{\deg}\,G=2n+1$.
Thus, if $\widehat{\deg}\,F=0$, then $\widehat{\deg}\,G=2n+1$, whereas if $\widehat{\deg}\,F=2n+1$, then $\widehat{\deg}\,G=0$.
Consequently, after interchanging $F(x)$ and $G(x)$ once more if necessary, we may assume that
\[
\widehat{\deg}\,F=0,\qquad\widehat{\deg}\,G=2n+1.
\]
We now prove \eqref{eq:forced_fn_eq} for this choice of $F$.

If $p$ is a prime divisor of $\Pi_n$, then $p\nmid 2n+1$, since $p<2n+1<2p$.
By \cref{claim:deg-equiv}, we have $p\mid G(0)$ for every such $p$, and hence $\Pi_n\mid G(0)$.
Moreover, $G(0)\neq 0$ since $F(0)G(0)=C_n\neq 0$.
By \cref{prop:location} and \cref{prop:Pi_n-ineq}, we obtain
\[
4^{\deg G}>|G(0)|\geq \Pi_n>4^{3n/5}.
\]
It follows that $\deg G>3n/5$, and hence $\deg F<2n/5$.

Set
\[
H(x)\coloneqq F(x+2)=x^{\deg F}+a_1x^{\deg F-1}+\cdots+a_{\deg F-1}x+a_{\deg F}.
\]
It remains to prove that $H(-x)=(-1)^{\deg F}H(x)$, which is equivalent to showing that $a_i=0$ for every odd $i$ with $1\leq i\leq\deg F$.

Let $p$ be a prime divisor of $\Pi_n$.
By \cref{claim:deg-equiv}, we have $p\nmid F(0)$, and hence $\overline{F}(x)\mid\overline{Q_{\frac{p-1}{2}}}(x)$.
It follows from \eqref{eq:local_fn_eq} that $\overline{H}(-x)=(-1)^{\deg F}\overline{H}(x)$.
Since $p$ is odd, this implies that $a_i\equiv 0\pmod{p}$ for every odd $i$.
Thus, $\Pi_n\mid a_i$ for every odd $i$.
By \cref{prop:location}, every root of $H(x)$ has absolute value less than $6$.
Hence, using \cref{prop:Pi_n-ineq}, we obtain
\[
|a_i|<\binom{\deg F}{i}6^i<(1+6)^{\deg F}<7^{2n/5}<2^{6n/5}<\Pi_n.
\]
Therefore, $a_i=0$ for every odd $i$ with $1\leq i\leq\deg F$.
This proves \cref{claim} and completes the proof of \cref{thm:main}.
\subsection*{Use of AI}
The author used ChatGPT, powered by GPT-6 Astra Pro, during the development of this work.
The key idea and an initial version of the proof of \cref{thm:main} were suggested by the model.
Through continued dialogue with ChatGPT, the author worked through the argument in detail and reorganized the proof and its exposition for clarity.
The author also used ChatGPT to formulate and refine the English prose of the manuscript based on content specified by the author.
The author takes full responsibility for all claims, proofs, and text in this paper.


\begin{thebibliography}{9}
\bibitem{Coleman}
R. F. Coleman, \emph{On the Galois groups of the exponential Taylor polynomials}, Enseign.~Math.~\textbf{33} (1987), 183--189.
\bibitem{KatzRivin}
N. M. Katz, I. Rivin, \emph{Speculations in Galois theory}, in \emph{Algebraic Geometry and the Langlands Program: A Volume in Honour of G\'erard Laumon}, Simons Symposia, Springer, Cham, to appear.
\bibitem{MattareiTauraso}
S. Mattarei, R. Tauraso, \emph{From generating series to polynomial congruences}, J. Number Theory \textbf{182} (2018), 179--205.
\bibitem{RosserSchoenfeld}
J. B. Rosser, L. Schoenfeld, \emph{Approximate formulas for some functions of prime numbers}, Illinois J. Math.~\textbf{6} (1962), 64--94.
\bibitem{Schur1929}
I. Schur, \emph{Einige S\"atze \"uber Primzahlen mit Anwendungen
auf Irreduzibilit\"atsfragen. I}, Sitzungsber.~Preuss.~Akad. Wiss.~Phys.-Math.~Kl.~(1929), 125--136.
\bibitem{Schur1930}
I. Schur, \emph{Gleichungen ohne Affekt}, Sitzungsber.~Preuss. Akad.~Wiss.~Phys.-Math.~Kl.~(1930), 443--449.
\end{thebibliography}
\end{document}